\documentclass[reqno]{amsart}

\usepackage{amssymb}
\usepackage{enumitem}
\usepackage{xcolor}
\usepackage{hyperref}

\newtheorem*{thmn}{Theorem}
\newtheorem*{mthmn}{Main theorem}
\newtheorem*{questn}{Question}

\newtheorem{thm}{Theorem}[section]
\newtheorem{cor}[thm]{Corollary}
\newtheorem{lem}[thm]{Lemma}
\newtheorem{prop}[thm]{Proposition}
\theoremstyle{definition}

\newtheorem{defn}[thm]{Definition}
\newtheorem{defnr}[thm]{Definition \& Remark}
\newtheorem{ex}[thm]{Example}
\newtheorem{rem}[thm]{Remark}

\newtheorem*{setting}{Main setting}

\numberwithin{equation}{section}

\newcommand{\itdef}[1]{\textit{#1}\index{#1}}
\newcommand{\itdefn}[1]{\itdef{#1}}

\newcommand{\bdefn}{\begin{defn}}
\newcommand{\edefn}{\end{defn}}
\newcommand{\bdefnr}{\begin{defnr}}
\newcommand{\edefnr}{\end{defnr}}
\newcommand{\benum}{\begin{enumerate}[label=(\arabic*),leftmargin=1.8em]}
\newcommand{\eenum}{\end{enumerate}}
\newcommand{\bthmn}{\begin{thmn}}
\newcommand{\ethmn}{\end{thmn}}
\newcommand{\bpropn}{\begin{propn}}
\newcommand{\epropn}{\end{propn}}
\newcommand{\bthm}{\begin{thm}}
\newcommand{\ethm}{\end{thm}}
\newcommand{\bnota}{\begin{nota}}
\newcommand{\enota}{\end{nota}}
\newcommand{\bproof}{\begin{proof}}
\newcommand{\eproof}{\end{proof}}
\newcommand{\bprop}{\begin{prop}}
\newcommand{\eprop}{\end{prop}}
\newcommand{\bcor}{\begin{cor}}
\newcommand{\ecor}{\end{cor}}
\newcommand{\blem}{\begin{lem}}
\newcommand{\elem}{\end{lem}}
\newcommand{\brem}{\begin{rem}}
\newcommand{\erem}{\end{rem}}
\newcommand{\bex}{\begin{ex}}
\newcommand{\eex}{\end{ex}}

\newcommand{\ba}{\begin{array}}
\newcommand{\ea}{\end{array}}
\newcommand{\bea}{\begin{eqnarray}}
\newcommand{\eea}{\end{eqnarray}}
\newcommand{\bean}{\begin{eqnarray*}}
\newcommand{\eean}{\end{eqnarray*}}

\newcommand{\cbb}{\mathbb{C}}

\newcommand{\lbb}{\mathbb{L}}
\newcommand{\nbb}{\mathbb{N}}

\newcommand{\rbb}{\mathbb{R}}

\newcommand{\CC}{\mathbb{C}}
\newcommand{\RR}{\mathbb{R}}

\newcommand{\ccal}{\mathcal{C}}

\newcommand{\ncal}{\mathcal{N}}

\newcommand{\qpscv}{$q$-pseudo\-concave}

\newcommand{\ph}{pluri\-harmonic}
\newcommand{\sph}{sub\-pluri\-harmonic}
\newcommand{\subph}{sub\-pluri\-harmonic}
\newcommand{\psh}{pluri\-sub\-harmonic}
\newcommand{\spsh}{strictly pluri\-sub\-harmonic}

\newcommand{\qpsh}{$q$-pluri\-sub\-harmonic}

\newcommand{\sqpsh}{strictly $q$-pluri\-sub\-harmonic}

\newcommand{\qpscvy}{$q$-pseudo\-concavity}

\newcommand{\pscv}{pseudo\-concave}
\newcommand{\hol}{holo\-morphic}
\newcommand{\qhol}{$q$-holo\-morphic}
\newcommand{\psc}{pseudo\-convex}
\newcommand{\spsc}{strictly pseudo\-convex}
\newcommand{\qpsc}{$q$-pseudo\-convex}
\newcommand{\pscy}{pseudo\-con\-vexity}
\newcommand{\qpscy}{$q$-pseudo\-con\-vexity}

\newcommand{\nbh}{neighborhood}

\newcommand{\cont}{continuous}
\newcommand{\usc}{upper semi-continuous}

\newcommand{\fct}{function}
\newcommand{\fcts}{functions}

\renewcommand{\and}{\ \mathrm{and}\ }
\newcommand{\qand}{\quad \mathrm{and} \quad}

\newcommand{\ol}[1]{\overline{#1}}

\newcommand{\relc}{\Subset}

\newcommand{\repa}{\mathrm{Re}}

\newcommand{\nach}{\rightarrow}

\newcommand{\eps}{\varepsilon}
\newcommand{\vphi}{\varphi}
\newcommand{\vrho}{\varrho}

\newcommand{\fe}{ \ \mathrm{for\ every}\ }

\newcommand{\D}{\displaystyle}
\newcommand{\cl}[1]{\ol{#1}}

\title[Foliations of continuous q-pseudoconcave graphs]{Foliations of continuous q-pseudoconcave graphs}

\author[T. Pawlaschyk]{Thomas Pawlaschyk}
\address{Thomas Pawlaschyk, School of Mathematics und Natural Sciences, University of Wuppertal, Gaussstr. 20, 42119 Wuppertal, Germany}
\email{pawlaschyk@uni-wuppertal.de}

\author[N. Shcherbina]{Nikolay Shcherbina}
\address{Nikolay Shcherbina, School of Mathematics und Natural Sciences, University of Wuppertal, Gaussstr. 20, 42119 Wuppertal, Germany}
\email{shcherbina@math.uni-wuppertal.de}

\keywords{q-plurisubharmonic functions, q-pseudoconcave sets,
singularity sets, foliations by holomorphic manifolds.}

\subjclass[2010]{Primary 32F10, 32D20; Secondary 37F75.}

\begin{document}
\begin{abstract} 

We show that for $k = 0, 1$ the graph of a continuous mapping $f:D \to \rbb^k\times\cbb^p$, defined on a domain $D$ in $\cbb^n\times\rbb^k$, is locally foliated by complex $n$-dimensional submanifolds if and only if its complement is $n$-pseudoconvex (in the sense of Rothstein) relatively to $(D\times\rbb^k)\times\cbb^p\subset \cbb^{n}\times\cbb^k\times\cbb^p$ . 

\end{abstract}

\maketitle

\section*{Acknowledgment} Research of the first author was supported by the Deutscher Akademischer Austauschdienst
(DAAD) and Conacyt M\'exico under the PPP Proalmex Project No. 51240052; and by the Deutsche Forschungsgemeinschaft (DFG) under the grant SH 456/1-1, {\it Pluripotential Theory, Hulls and
Foliations}. The authors wish to thank the anonymous referee for valuable comments and suggestions which improved the presentation of the paper. This version of the article has been accepted for publication in Indiana University Mathematics Journal, after peer review but is not the final published version and does not reflect post-acceptance improvements, or any corrections. The final version is available online at: \href{https://doi.org/10.1512/iumj.2022.71.9010}{https://doi.org/10.1512/iumj.2022.71.9010}.


\section{Introduction}

The notion of pseudoconvexity and the problem of existence of a complex structure on a given set play important roles in complex analysis. That is why it is natural to look for a possible link between them and to raise the following (maybe a bit naive) question.



\begin{questn} Let $U \subset V$ be two domains in ${\mathbb C}^n$, $ n \geq 2$, and let $U$ be pseudoconvex relatively to $V$, i.e., for every point $z$ in $V \cap \partial U$ there is a ball $B_r(p)$ centered at $p$ such that $U \cap B_r(p)$ is \psc. Does it follow that the set $\Gamma := V \setminus U$ possesses a holomorphic structure?

Here, by the existence of a holomorphic structure on a locally closed set $\Gamma \subset {\mathbb C}^n$ we mean that for every point $z \in \Gamma$ there is a complex analytic variety $ A_z \ni z$ of positive dimension such that $A_z \subset \Gamma$.

\end{questn}

The answer to this question is, in general, negative. Examples, in the case $n=2$, of a ``thin'' (of topological dimension two) and ``thick'' (of topological dimension three) sets $\Gamma$ with pseudoconvex complements but without holomorphic structure were constructed in \cite{Wermer} and  \cite{Sh82}, respectively. 

However, if we restrict ourselves to a more specific choice of the set $\Gamma$, namely, to the case when $\Gamma$ has the structure of a graph, then the situation will become rather different. The first classical and fascinating result in this direction for graphs of real codimension two which are merely continuous is due to Hartogs \cite{Har2} and comes back to 1909.

\begin{thm}\label{thm-hartogs-pscv}  Let $f:B\to\cbb_\zeta$ be a continuous function defined on the unit ball $B \subset\cbb^n_z$. Then the complement of its graph $\Gamma(f)=\{(z,\zeta): z \in B, \zeta=f(z)\}$  in $B\times\cbb_\zeta$ is a domain of holomorphy if and only if the function $f$ is holomorphic. 
\end{thm}

A similar statement holds true for continuous graphs of real codimension one. 

\begin{thm}\label{thm-levi-pscv} Let $f:B'\to\rbb_v$ be a continuous function defined on the unit ball $B'$ in $\cbb^n_z\times\rbb_u$. Then the complement of its graph $\Gamma(f)=\{(z,w):(z,u) \in B',\ v = f(z,u),\ w=u+iv\}$ in $(B'\times\rbb_v) \subset \cbb^n_z\times\cbb_w$ is a disjoint union of two domains of holomorphy if and only if the set $\Gamma(f)$ can be decomposed into a disjoint union of complex analytic hypersurfaces.
\end{thm}

The last result was first proved in a fundamental paper of Levi \cite{Levi} from 1911 for smooth graphs in the special case $n=1$.  It was established much later in \cite{Sh93} that if we drop the smoothness assumption on $f$ and only demand continuity, then the statement of Theorem~\ref{thm-levi-pscv} still holds true in the case $n=1$. Later, based on this result, the same statement for continuous $f$  was proved in \cite{Ch} in the case $n \geq 2$.

\begin{rem} More details on the theorems of Hartogs and Levi and their complete proofs can be found in  \cite[Theorem 2 on p.226]{Shabat} and \cite[Levi's theorem on p.164]{Vladimirov}, respectively.
\end{rem}

The main purpose of the present paper is to provide a generalization of  Theorems~\ref{thm-hartogs-pscv} and ~\ref{thm-levi-pscv} to the case of graphs of higher codimension, i.e., to the case when the target space has dimension larger than one. In this case the requirement of pseudoconvexity of the complement of $\Gamma(f)$ is not a proper condition any more. Indeed, in view of the Hartogs' theorem on removability of compact singularities (see, for example, \cite[Theorem 3 on p.172]{Shabat}) applied to the vertical fibers in the case of a higher codimensional generalization of  Theorem~\ref{thm-hartogs-pscv},  or, in view of Kontinuit\"atssatz (see, for example, \cite[Chapter III, Section 17]{Vladimirov}) in the case of a higher codimensional generalization of  Theorem~\ref{thm-levi-pscv}, the complement of the set $\Gamma(f)$ can never be pseudoconvex. However, it turns out that if we substitute the condition of pseudoconvexity of the complement by its $n$-pseudocovexity in the sense of Rothstein \cite{Ro} (for the definition and more details on this notion see Section 3), then the following generalization of Theorems~\ref{thm-hartogs-pscv} and ~\ref{thm-levi-pscv} holds true (see Theorem~\ref{thm-pscv-fol} below).

\begin{mthmn} Let $n,k,p$ be integers with $n \geq 1$, $p \geq 0$ and let $k \in \{0,1\}$ such that $N=n+k+p \geq 2$. Let $B''$ be the unit ball in $\cbb^n_z\times\rbb^k_u$ and let $f: B'' \rightarrow \rbb^k_v \times \cbb^p_\zeta$ be a \cont \ mapping. Then the complement of its graph $\Gamma(f)$ in $B''\times\rbb^k_v \times \cbb^p_\zeta \subset \cbb^N_{z,u+iv,\zeta}$ is  $n$-\psc{} in the sense of Rothstein if and only if $\Gamma(f)$ is foliated by $n$-dimensional complex submanifolds.
\end{mthmn}

Notice that our new result corresponds to the case $n=k=p=1$, while the rest is either already known, as mentioned before, or can be derived from it. Results in the case $n\geq1$, $k \geq 2$, $p \geq 0$ are known to us only in the smooth setting (see Proposition \ref{prop-k2}). Notice also that $q$-pseudoconvexity of sets in $\cbb^n$ is the same as $(n-q-1)$-pseudoconvexity in the sense of S{\l}odkowski \cite{Sl2} and $(n-q)$-convexity in the sense of Grauert and Andreotti \cite{AG} (in the smooth setting).

The paper is organized as follows. In Sections \ref{chap-qpsh} and \ref{chap-qpsc} we collect some (mainly known) facts on $q$-\psh\ \fcts\ and \qpsc\ sets. In Section \ref{chap-duality} we present results on the dual relation between Levi \qpsc\ sets and (non-)\qpsc\ sets, which will form the essential tools to prove the main theorem. In the last Section \ref{chap-main} we prove our main result, Theorem~\ref{thm-pscv-fol}.

These results were obtained jointly by the authors of this paper and were included in the doctoral thesis \cite{TPTHESIS} of the first author. They have not been published yet, since we intended to proceed with further investigations in hope to prove more general statements. Meanwhile, Takeo Ohsawa generalized Nishino's rigidity theorem \cite{Nishino} using $L^2$-methods. One of his generalizations states that an $n$-convex complex manifold $M^{n+m}$ fibered over $\mathbb{D}^m$ $\big(\mathbb{D}=\{z \in \cbb : |z|<1\}\big)$ is biholomorphically equivalent to $\mathbb{D}^m\times \cbb\mathbb{P}^n_*$ if the fibers are equivalent to $\cbb\mathbb{P}^n_*$, i.e.~once-punctered $\cbb\mathbb{P}^n$ (see \cite{Oh}). As an immediate consequence he derives a statement similar to ours in the case $k=0$. As it appears that our result is more general ($k=0,1$) and its verification uses different techniques, we felt motivated to finalize our paper and present our results in its original form given in the above mentioned thesis.


\section{On $q$-Plurisubharmonic Functions}\label{chap-qpsh}

In this section, we introduce the \qpsh\ \fcts\ in the sense of Hunt and Murray \cite{HM}. A collection of results on \qpsh\ \fcts\ can be found in \cite{Dieu} and in the doctoral thesis \cite{TPTHESIS} of the first author. 

\bdefn \label{def-qpsh}\label{defn-sph} Let $q\in\{ 0,\ldots,n-1\}$ and let $\psi$ be an \usc\ \fct\ on an open set $\Omega$ in $\cbb^n$.
\benum

\item The \fct\ $\psi$ is called \itdefn{\sph} \textit{on $\Omega$} if for every ball $B \relc U$ and every \fct\ $h$ which is \ph\ on a \nbh\ of $\cl{B}$ with $\psi \leq h$ on $\partial B$ we already have that $\psi \leq h$ on $\cl{B}$.

\item The \fct\ $\psi$ is \textit{\qpsh}\ on $\Omega$ if $\psi$ is \sph\ on $\Pi \cap \Omega$ for every complex affine subspace $\Pi$ of dimension $q+1$.


\item If $q \geq n$, then every \usc\ \fct\ on $\Omega$ is by convention $q$-\psh.


\item An \usc\ \fct\ $\psi$ on $\Omega$ is called \textit{strictly \qpsh}\ \textit{on $\Omega$} if for every $\ccal^\infty$-smooth non-negative \fct\ $\theta$ with compact support in $\Omega$ there is a positive number $\eps_0$ such that $\psi+\eps\theta$ remains \qpsh\ on $\Omega$ for every real number $\eps$ with $|\eps|\leq\eps_0$.

\eenum
\edefn

Now we provide a non-trivial example of a \subph\ \fct. Examples of smooth \qpsh\ \fcts\ can be easily derived from Theorem~\ref{smooth-qpsh} below.

\bex \label{ex-sph-norm} \label{ex:subph}

Let $\Omega$ be an arbitrary domain in $\CC^n$. Then it was proved in \cite[Example 2.8]{TPESZ2} that for any complex norm $\|\cdot\|$ on $\CC^n$ the \fct
\[
\CC^n\setminus\{0\} \ni z \mapsto -\log \|z\|
\]
and the induced boundary distance function 
\[
\Omega \ni z\mapsto d_{\|\cdot\|}(z,\partial \Omega) := -\log \inf\{\|z-w\|: w \in \partial\Omega\}
\]
are always \sph.
\eex

We give a list of properties of \qpsh\ \fcts.

\bprop \label{prop-qpsh} Every below mentioned \fct\ is defined on an open set $\Omega$ in $\cbb^n$ unless otherwise stated.
\benum

\item \label{prop-qpsh-psh-sph} The $0$-\psh\ \fcts\ are exactly the \psh\ \fcts, and the $(n-1)$-\psh\ \fcts\ are the \sph\ \fcts.

\item \label{prop-qpsh-q+1-psh} Every \qpsh\ \fct\ is $(q+1)$-\psh.

\item \label{qpsh-sum}\cite{Sl2} If $\psi$ is \qpsh\ and $\varphi$ is $r$-\psh, then $\psi + \varphi$ is $(q+r)$-\psh.



\item  \label{prop-qpsh-local} \cite{HM} An \usc\ \fct\ $\psi$ is \qpsh\ on $\Omega$ if and only if it is locally \qpsh\ on $\Omega$, i.e., for each point $p$ in $\Omega$ there is a \nbh\ $U$ of $p$ in $\Omega$ such that $\psi$ is \qpsh\ on $U$.

\item \label{thm-equiv-wqpsh} \cite{Fu2, Dieu} A \fct\ $\psi$ is \qpsh\ on an open set $\Omega$ in $\cbb^n$ if and only if $\psi \circ f$ is \qpsh\ for every holomorphic mapping $f:D \to \Omega$, where $D$ is a domain in $\cbb^{q+1}$ (or even $\cbb^{k}$ with $k \geq q+1$).



\item \cite{HM}\label{prop-qpsh-patch} Let $\Omega_1$ be an open set in $\Omega$, $\psi$ be a \qpsh\ \fct\ on $\Omega$ and $\psi_1$ be a \qpsh\ \fct\ on $\Omega_1$ such that
\[
\limsup_{\substack{w\nach z\\ w \in \Omega_1}}\psi_1(w) \leq \psi(z) \fe z \in \partial\Omega_1 \cap \Omega.
\]
Then the subsequent \fct\ is \qpsh\ on $\Omega$,
\[
\vphi(z):=\left\{\ba{ll} \max\{ \psi(z),\psi_1(z)\}, & z \in \Omega_1\\ \psi(z), & z \in \Omega\setminus \Omega_1 \ea\right. .
\]

\eenum
\eprop

A smooth (strictly) \qpsh\ \fct\ can be characterized by counting the eigenvalues of its complex Hessian matrix.

\defn \label{defn-levi-form} Let $\psi$ be twice differentiable at a point $p$. For $X,Y \in \cbb^n$ we define the \textit{Levi form}\index{Levi!form} \textit{of $\psi$ at $p$} by
\[
\mathcal{L}_\psi(p)(X,Y):=\sum_{k,l=1}^n \frac{\partial^2 \psi}{\partial z_k \partial \ol{z}_l}(p)X_k \ol{Y}_l.
\]

We have the following characterization of smooth \qpsh\ \fcts\ (see \cite{HM}). A generalization to almost everywhere twice differentiable \qpsh\ \fcts\ is due to S{\l}odkowski \cite{Sl2}.

\bthm\label{smooth-qpsh} Let $q \in \{0,\ldots,n-1\}$ and let $\psi$ be a $\ccal^2$-smooth \fct\ on an open subset $\Omega$ in $\cbb^n$. Then $\psi$ is (strictly) \qpsh\ if and only if the Levi form $\mathcal{L}_\psi(p)$ has at most $q$ negative ($q$ non-positive) eigenvalues at every point $p$ in $\Omega$.
\ethm

In the same paper \cite{HM} the local maximum property was shown.

\begin{thm}[Local maximum property] \label{prop-qpsh-locmax} Let $q\in\{ 0,\ldots,n-1\}$ and $\Omega$ be a relatively compact open set in $\cbb^n$. Then any \fct\ $\psi$ which is \usc\ on $\cl{\Omega}$ and \qpsh\ on $\Omega$ fulfills
\[
\max_{\cl{\Omega}} \psi= \max_{\partial\Omega} \psi.
\]
\end{thm}

S{\l}odkowski generalized the previous results to analytic sets in Proposition~5.2 and Corollary~5.3 of \cite{Sl}. We will need both of them later on but first, we recall the notion of the dimension of an analytic subset $A$ in the complex Euclidean space. For more details we refer to \cite{Chirka}.

\bdefn Let $A$ be an analytic subset of $\CC^n$. Given a point $p$ in $A$, the \textit{dimension of $A$ at $p$} is defined by
\[
\dim_p A:=\limsup_{\substack{z \to p \\ z \in A^{*}}} \ \dim_z A,
\]
where $A^{*}$ denotes the set of all regular points of $A$ and $\dim_z A$ is the dimension of the complex manifold $A \cap U$ in a \nbh\ $U$ of the regular point $z \in A^{*}$. 
\edefn

\begin{thm}[Local maximum principle for analytic sets]\label{prop-loc-max-analyt} Fix an integer number $q \in \{0,\ldots,n-1\}$. Let $A$ be an analytic subset of an open set $\Omega$ in $\cbb^n$ with $\dim_z A \geq q+1$ for all $z \in A$ and let $\psi$ be a \qpsh\ \fct\ on $A$, i.e. for every point $z \in A$ the \fct\ $\psi$ extends to a \qpsh\ \fct\ on some open \nbh\ of $z$ in $\Omega$. Then for every compact set $K$ in $A$ we have that
\[
\max_K \psi = \max_{\partial_AK} \psi.
\]
Here, by $\partial_A K$ we mean the relative boundary of $K$ in $A$.
\end{thm}


\section{On $q$-Pseudoconvex Sets}\label{chap-qpsc}

Several characterizations of \qpscy\ in $\cbb^n$ can be found in the literature. We may refer, for example, to \cite{Ro}, \cite{Fu0}, \cite{EVS}, \cite{Sl} and \cite{Ma}. Out of these notions we need the \qpscy\ in the sense of Rothstein \cite{Ro} which is based on generalized Hartogs figures. Later, it has been studied by \cite{EVS} and \cite{Su} who named it \textit{Hartogs \qpscy}. To distinguish the different notions of \qpscy\ in our paper we keep this notation.

\bdefn\label{HaFi}

\benum
\item  We write $\Delta^n_r:=\Delta_r^n(0)=\{z\in\cbb^n:\max_j|z_j|<r\}$ for the polydisc with radius $r>0$ and $A^n_{r,R}:=\Delta^n_R\setminus\cl{\Delta^n_r}$.

\item Let $n\geq 2$ and $q \in \{1,2,\ldots,n-1\}$ be fixed integers, and let $r$ and $R$ be real numbers in the interval $(0,1)$. An
\emph{Euclidean $(n-q,q)$ Hartogs figure} $H_e$ is the set
\bean
H_e:=\big(\Delta^{n-q}_1\times\Delta^q_r\big) \cup \big(A^{n-q}_{R,1}\times\Delta^q_1\big) \ \subset \ \Delta^{n-q}_1 \times \Delta^q_1=\Delta^n_1.
\eean


\item A pair $(H,P)$ of domains $H$ and $P$ in $\cbb^n$ with $H \subset P$ is called a \textit{(general) $(n-q,q)$ Hartogs figure} if there is an Euclidean $(n-q,q)$ Hartogs figure~$H_e$ and a biholomorphic mapping $F$ from $\Delta^n_1$ onto $P$ such that $F(H_e)=H$.

\item An open set $\Omega$ in $\cbb^n$ is called \textit{Hartogs $q$-pseudoconvex} if it admits the Kontinuit\"atssatz with respect to the $(n-q)$-dimensional polydiscs, i.e., given any $(n-q,q)$ Hartogs figure $(H,P)$ such that $H\subset \Omega$, we already have that $P\subset\Omega$.

\eenum
\edefn




Notice that in regards to the classical Kontinuit\"atssatz, a set in $\CC^n$ is Hartogs $(n-1)$-pseudoconvex if and only if it is pseudoconvex.


A proof of the following statement and an extended list of notions which are equivalent to $q$-pseudoconvexity can be found in \cite{TPESZ2} or in \cite{TPTHESIS}.

\begin{thm}\label{equivqpsc}\label{prop-qpsc}\label{eqivqpsc}
Let $q \in \{1,2,\ldots,n-1\}$ and $\Omega$ be an open set in $\cbb^n$. Then the following statements are all equivalent.
\benum

\item\label{eqivqpsc1} $\Omega$ is Hartogs $q$-\psc.




\item\label{eqivqpsc5} $\Omega$ is $(n-q-1)$-\psc\ in the sense of S{\l}odkowski \cite{Sl}, i.e., the function $z \mapsto -\log d(z,\partial\Omega)$ is $(n-q-1)$-\psh\ on $\Omega$. Here, $d(z,\partial\Omega)$ denotes the Euclidean distance from $z$ to the boundary of $\Omega$.

\item \label{eqivqpsc3} For every complex norm $\|\cdot\|$ on $\CC^n$ the \fct\ 
\[
z \mapsto -\log d_{\|\cdot\|}(z,b\Omega)=-\log \inf\{\|z-w\|: w \in b\Omega\}
\]
is $(n-q-1)$-\psh\ on $\Omega$.

\item\label{eqivqpsc4} There exists an $(n-q-1)$-\psh\ exhaustion \fct\ $\Phi$ on $\Omega$, i.e., for every $c \in \rbb$ the set $\{z \in \Omega : \Phi(z) < c\}$ is relatively compact in $\Omega$.


\item\label{eqivqpsc7} Let $\{A_t\}_{t\in[0,1]}$ be a family of analytic subsets in some open set $U$ in $\cbb^n$ that continuously depend on $t$ in the Hausdorff topology. Assume further that $\dim_z A_t \geq n-q$ for every $z \in A_t$ and $t \in [0,1]$, and that the closure of $\bigcup_{t\in[0,1]}A_t$ is compact. If $\Omega$ contains the boundary $\partial A_1$ and the closure $\overline{A_t}$ for each $t\in[0,1)$, then the closure $\cl{A_1}$ also lies in $\Omega$.

\item\label{eqivqpsc9} For every point $p$ in $\partial\Omega$ there is a ball $B_r(p)$ centered at $p$ such that $\Omega \cap B_r(p)$ is \qpsc.



\eenum
\end{thm}



\brem The equivalence of \ref{eqivqpsc5}, \ref{eqivqpsc7} and \ref{eqivqpsc9} (and even more) already has been verified by S{\l}odkowski \cite[Chapter 4]{Sl}. He also showed that every domain $\Omega$ in $\CC^n$ admits an $(n-1)$-\psh\ exhaustion function and, therefore, is $(n-1)$-\psc\ in his sense. The equivalent property \ref{eqivqpsc3} has been shown in \cite[Propsition 3.3]{TPESZ2}.\erem

In the next section, we will examine subdomains which are \qpsc\ only in a \nbh\ of particular points of their boundary relative to a prescribed larger domain. For this, we introduce the notion of \textit{relative Hartogs \qpscy}\ inspired by \cite[Definition 4.1]{Sl}.

\bdefn \label{defn:relqpsc} Let $U\subsetneq V$ be two domains in $\CC^n$. We say that a $U$ is \textit{Hartogs \qpsc\ relatively to $V$} if for every point $p$ in $V\cap{\partial U}$ there exists an open ball $B_r(p)$ centered in $p$ such that the intersection $U\cap{B}_r(p)$ is Hartogs \qpsc.
\edefn

The next characterization of relative \qpsc\ sets is similar to properties~\ref{eqivqpsc1} and~\ref{eqivqpsc5} of Theorem~\ref{equivqpsc}.

\begin{prop}\label{equiv-rqpsc}
Let $U\subsetneq V$ be open sets in $\cbb^n$. Then the following statements are equivalent.

\benum

\item \label{equiv-rqpsc2} $U$ is Hartogs \qpsc\ relatively to $V$.

\item \label{equiv-rqpsc1} There exist a \nbh\ $W$ of $\partial U \cap V$ in $V$ and a $(n-q-1)$-\psh\ \fct\ $\psi$ on $W \cap U$ such that $\psi(z)$ tends to $+\infty$ whenever $z$ approaches the relative boundary $\partial U \cap V$.
\eenum
\end{prop}

In the case of $V=\cbb^n$, the Hartogs \qpsc\ sets relatively to $\cbb^n$ are simply the Hartogs \qpsc\ sets defined in the beginning of this chapter. Moreover, we have the following examples.

\bex\label{ex-qpsc-sub} 
\benum

\item \label{ex-qpsc-sub-1} Let $\vphi$ be $(n-q-1)$-\psh\ on an open set $V$ in $\cbb^n$ and let $c$ be a real number. Then, by Theorem~\ref{prop-loc-max-analyt}, the set $U=\{z\in V:\vphi(z)<c\}$ is Hartogs \qpsc\ relatively to $V$. If, moreover, the set $V$ is itself Hartogs \qpsc, then $U$ is also Hartogs \qpsc\ (in $\CC^n$) in view of Definition~\ref{HaFi}.

\item \label{ex-qpsc-sub-2} Let $\Omega$ be an open set in $\cbb^n$ and let $h$ be a smooth \qhol\ \fct\ on~$\Omega$ in the sense of Basener \cite{Ba}, i.e., $\overline{\partial} h \wedge (\partial \overline{\partial} h)^{q}=0$. Let $\Gamma(h):=\{(z,h(z)) \in \cbb^{n+1}: z \in \Omega\}$ be the graph of $f$ over $\Omega$. Then the \fct\ $(z,w)\mapsto 1/(h(z)-w)$ is \qhol\ on $U:=(\Omega\times\cbb)\setminus\Gamma(h)$ by \cite{Ba}, so the function $\psi(z,w):=-\log|h(z){-}w|$ is \qpsh\ on $U$ by \cite{HM}. It has the property that $\psi(z,w)$ tends to $+\infty$ whenever $(z,w)$ approaches the graph $\Gamma(h)$. Hence, the open set $U$ is Hartogs $(n-q)$-\psc\ relatively to $V:=\Omega\times\cbb \subset \CC^{n+1}$ by Proposition~\ref{equiv-rqpsc}~\ref{equiv-rqpsc1}. A converse statement is not known except for Hartog's theorem which appears in the holomorphic case $q=0$.
\eenum

\eex

We also need the smoothly bounded \qpsc\ sets which were studied in \cite{Ba}, \cite{HM}, \cite{EVS} and \cite{Su}.

\begin{defn}\label{defn-levi-qpsc} Let $U$ be an open set in $\cbb^n$.
\benum

\item The set $U$ is called \textit{Levi \qpsc}\ (resp. \textit{strictly (Levi) \qpsc}) \textit{at the point $p\in\partial U$} if there exist a \nbh\ $W$ of $p$ and a $\ccal^2$-function $\vrho$ on $W$ such that $\nabla \vrho(p) \neq 0$, $U\cap W = \{z \in W: \vrho(z)<0\}$ and, moreover, such that its Levi form $\mathcal{L}_\vrho$ at $p$ has at most $q$ negative (resp. $q$~non-positive) eigenvalues on the holomorphic tangent space 
\[
H_p \partial U=\{X \in \cbb^n : \sum_{j=1}^n \frac{\partial \vrho}{\partial z_j}(p)X_j=0\}.
\]
(Clearly, the definition of Levi $q$-pseudoconvexity does not depend on the choice of the defining function $\vrho$.)

\item Let $V$ be an open \nbh\ of $U$ with $U \subsetneq V$. Then $U$ is called \textit{(strictly) Levi \qpsc} \textit{in $V$} if it is (strictly) Levi \qpsc\ at every point $p\in \partial U \cap V$. A strictly Levi \qpsc\ set in $\cbb^n$ is also simply called \textit{strictly \qpsc}.

\eenum
\end{defn}

We present some facts about Levi \qpsc\ sets.

\begin{rem} \label{rem-def-fct}
If $U \subset \cbb^n$ is strictly \qpsc\ at a boundary point $p$ and $\psi$ is a defining \fct\ for $U$ at $p$, then for a large enough constant $c>0$ the \fct\ $\exp(c\psi)-1$ is strictly \qpsh\ on some ball $B$ centered at $p$ and still defines $U$ at $p$. In addition, if $U$ is bounded, there exists a global strictly \qpsh\ defining \fct\ $\vrho$ on some \nbh\ of $\overline{U}$ such that $U=\{\vrho <0\}$. This means that $U$ is Hartogs $(n-q-1)$-\psc\ according to Example~\ref{ex-qpsc-sub}~\ref{ex-qpsc-sub-1}.
\end{rem}

The next theorem is due to \cite{EVS} and \cite{Su} and is reformulated in terms of Hartogs \qpscy\ after application of Theorem~\ref{equivqpsc}.

\begin{thm} \label{thm-suria} Let $U$ be a $\mathcal{C}^2$-smoothly bounded domain in $\cbb^n$. Then $U$ is  Levi \qpsc\ if and only if it is Hartogs $(n-q-1)$-\psc.
\end{thm}


\section{Duality Principle of $q$-Pseudoconvex Sets}\label{chap-duality}

In this section, we study the link between strictly \qpsc\ sets and their complements. Their relation leads to two duality theorems. The first one is due to Basener (see \cite[Proposition 6]{Ba}).

\begin{thm}\label{lemCofspsc}
If an open set $\Omega$ in $\cbb^n$ is strictly \qpsc\ at some point $p\in{\partial\Omega}$, then for every small enough \nbh\ $V$ of $p$, for each $w$ in $\partial\Omega\cap{V}$ and every \nbh\ $W\relc{V}$ of $w$ there is a family $\{A_t\}_{t \in [0,1]}$ of $(n{-q-}1)$-dimensional complex submanifolds of $W$ which is continuously parameterized by $t$ and fulfills
\benum
\item $\cl{A}_t\subset(\cbb^n\setminus\cl{\Omega})$ for every $t \in [0,1)$,
\item $w\in{A_1}$, but $\cl{A}_1\setminus\{w\}\subset(\cbb^n\setminus\cl{\Omega})$.
\eenum
In view of Theorem~\ref{equivqpsc} above, this means that $U:=V \cap (\cbb^n\setminus\overline{\Omega})$ is not Hartogs $(q+1)$-\psc\ relatively to $V$.
\end{thm}

\begin{proof} Fix the point $p\in\partial \Omega$. By Proposition 6 in \cite{Ba}
there exists a \nbh\ $V$ of $p$ in $\cbb^n$ such that for every point $w$ in $\partial\Omega\cap{V}$ the following
properties hold after an appropriate holomorphic change of coordinates on $V$,
\begin{equation}\label{eqlemCofspsc01}
w=0 \qand \mathrm{Re}(z_1) < 0 \ \fe z \in V \cap (\cl{\Omega}\setminus \{w\})\cap(\cbb^{n{-}q}{\times}\{0\}^q).
\end{equation}
Let $U\relc{V}$ be any \nbh\ of $w$. Then there are real numbers $\eps>0$ and $r>0$ such that  for each $t \in [0,1]$
the submanifold \[A_t=\{(1-t)\eps\}{\times}B^{n-q-1}_r(0){\times}\{0\}^q\] is contained in $U$. Finally, the properties~\eqref{eqlemCofspsc01} imply that the family $\{A_t\}_{t\in[0,1]}$ has the desired properties.
\end{proof}

In order to establish a converse statement of Theorem \ref{lemCofspsc}, we need the following lemma.

\blem \label{lem-sb} Let $q \in \{0,\ldots,n-1\}$ and let $\psi$ be a $\ccal^2$-smooth strictly \qpsh\ \fct\ on an open set $V$ in $\cbb^n$. Assume that $V$ contains two compact sets $K$ and $L$ which fulfill the following properties:
\begin{itemize}
\item[(1)] $K,L \subset \{z \in V : \psi(z) \leq 0\}$
\item[(2)] $L \cap \{z \in V : \psi(z) = 0\}=\emptyset$
\item[(3)] $K \cap \{z \in V : \psi(z) = 0\} \neq \emptyset$
\end{itemize}
Under these conditions, there exist a point $z_0 \in \partial K$, a \nbh\ $U \relc V$ of $z_0$ and a $\ccal^2$-smooth \sqpsh\ \fct\ $\vphi$ on $U$ satisfying:
\begin{itemize}

\item[(a)] $K,L \subset \{z \in U : \vphi(z) \leq 0\}$
\item[(b)] $L \cap \{z \in U : \vphi(z) = 0\}=\emptyset$
\item[(c)] $K \cap \{z \in U : \vphi(z) = 0\} =\{z_0\}$
\item[(d)] $\nabla\vphi \neq 0$ on $\{z \in U : \vphi(z)=0\}$

\end{itemize}
In other words, the set $G:=\{z \in U : \vphi(z)<0\}$ is strictly \qpsc\ at points of the set $\{z \in U: \vphi(z)=0\}$. Moreover, it contains $L$, and $K$ touches $\partial G$ from the inside of $G$ only at the point~$z_0$.
\elem

\bproof We proceed similarly to the proof of Proposition 3.2 in \cite{HST}. Let $\delta>0$ and $V_{\delta}:=B_{1/\delta}(0)\cap\{z \in V : d(z,\partial V) > \delta\}$. We choose $\delta>0$ so small that the conditions (1) to (3) of this lemma still hold if we replace $V$ by $V_\delta$.

We set $U:=V_\delta$. Let $B:=B_\delta(0)$ and consider the \fct\ $f:B\to\psi(U)$ defined by $f(w):=\max_{z \in K} \psi(z+w)$. Pick a point $p \in K \cap \{z \in U: \psi(z)=0\}$. Since $\psi$ is \sqpsh, it follows from the local maximum property (see Theorem \ref{prop-qpsh-locmax}) that $\{\psi>0\} \cap W$ is not empty for any \nbh\ $W$ of $p$. Hence, since $p$ belongs to $K$ and since $f(0)=\psi(p)=0$, the image $f(B)$ contains a non-empty open interval $I=(0,\delta')$ for some $\delta'>0$. Since $f(B)$ lies in $\psi(U)$, Sard's theorem implies that there exists a regular value $f(w_0)$ inside $I$ which is so close to $\psi(p)=0$ that the conditions (1) to (3) are still valid for the \fct\ $\psi_0(z):=\psi(z+w_{0})-f(w_0)$ instead of $\psi$. Notice that $0$ is now a regular value for $\psi_0$.

Let $z_0$ be a point in $K$ with $f(w_0)=\psi(z_0+w_0)$, so that $\psi_0(z_0)=0$. For $\eps>0$, we define $\vphi(z):=\psi_0(z)-\eps|z-z_0|^2$. Then it is easy to see that the set $K \cap \{z \in U : \vphi(z)=0\}$ contains only the point $z_0$, so we also obtain property (c). Now if $\eps>0$ is small enough, then the function $\varphi$ is still strictly $q$-\psh. Besides of that, the \fct\ $\vphi$ fulfills also the properties (a) and (b). Finally, by the choice of $f(w_0)$, zero is a regular value for $\vphi$, so we also gain the property (d).
\eproof

The next result is the second duality theorem and a converse statement to Theorem~\ref{lemCofspsc}.

\begin{thm}\label{lemNotpsc}
Let $\Omega$ be a domain in $\cbb^n$ which is not Hartogs $q$-\psc. Then there exist a point $p\in\partial \Omega$, a \nbh\ $V$ of $p$ and a strictly $(q-1)$-\psc\ set $G$ in $V$ such that the set $V\setminus{\Omega}$ is contained in $G\cup\{p\}$ and $\{p\}=\partial G\cap\partial \Omega\cap{V}$, i.e., $V \setminus \Omega$ touches $\partial G$ from the inside of $G$ only at $p$.
\end{thm}

\begin{proof} Since $\Omega$ is not Hartogs \qpsc, then, in view of Definition \ref{HaFi}, there exists a $(n-q,q)$-Hartogs figure $(H,P)$ and a biholomorphic mapping $F$ on $\Delta:=\Delta^n_1(0)$ onto its image in $\cbb^n$ such that $H=F(H_e)$ lies in $\Omega$, but $P=F(\Delta^n_1(0))$ is not contained entirely in $\Omega$ for the Euclidean Hartogs figure
\[
H_e=\big(\Delta^{n-q}_1\times\Delta^{q}_r\big) \cup \big(A^{n-q}_{R,1}\times\Delta^{q}_1\big) \ \subset \ \cbb^{n-q}_{z}\times\cbb^{q}_{w}.
\]
After shrinking $\Delta$ if necessary, we can assume that $F$ is defined on a \nbh\ of the closure of $\Delta$. We set $M:=F^{-1}(\cbb^n\setminus\Omega) \cap \cl{\Delta}$. Now let $\alpha,\beta \in (0,1)$ and set
\[
K_0:=\big(\cl{\Delta^{n-q}_1 \times \Delta^{q}_{\alpha}}\big) \cap M \qand L_0:=\big(\cl{\Delta^{n-q}_1 \times A^{q}_{\beta,1}}\big) \cap M.
\]
Since $\Phi(H_e)$ lies in $\Omega$, we can find an appropriate $\alpha \in (0,1)$ such that $K_0$ is not empty. Fix some $\beta \in (0,1)$ with $\alpha < \beta$. Let $|w|_\infty=\max_{j=1,\ldots,q}|w_j|$ and consider the \fct\ $u(w):=-\log|w|_\infty$. By the assumptions made on $H$ and $P$, we can find a large enough number $c \in \rbb$ such that
\bea\label{eq-dual-001}
& M \subset \ D_c(u):=\{(z,w) \in \cl{\Delta} : u(w)<c\}. &
\eea
 Let $k \in \nbb$ and define the \fct\ $u_k$ by
\[
u_k(w):=-\frac{1}{k}\log|(w_1^k,\ldots,w_{q}^k)| + \frac{1}{k}|w|^2,
\]
where $|w|$ denotes the euclidean norm of the vector $w = (w_1,w_2, \ldots,w_q)$.
Then the \fct\ $u_k$ is $\ccal^\infty$-smooth and strictly $(q-1)$-\psh\ on $\cbb_w^{q}\setminus\{0\}$ (see \cite[Example 2.8 ]{TPESZ2} or Example~\ref{ex:subph} above). Moreover, the sequence $(u_k)_{k \in \nbb}$ converges  to $u$ uniformly on compact sets in $\cbb_w^{q}\setminus\{0\}$. Therefore, and in view of property~\eqref{eq-dual-001}, we can pick an integer $k_0$ so large that $M$ lies in $D_c(u_k):=\{(z,w) \in \cl{\Delta} : u_k(w)<c\}$ for every $k \geq k_0$. Define
\[
c_k:=\inf\big\{ a \in \rbb : M \subset D_a(u_k)\big\}.
\]
Now, since $\alpha < \beta$, we can fix an even larger $k \geq k_0$ so that $L_0 \subset D_{c_k}(u_k)$ and $L_0 \cap \{(z,w) \in \cl{\Delta}: u_k(w)=c_k\}=\emptyset$. Then it is easy to see that $K_0$ intersects $\{(z,w) \in \cl{\Delta} : u_k(w)=c_k\}$ in a point $\zeta_0 \in \Delta$. Finally, we set $U:=F^{-1}(\Delta)$, $K:=F^{-1}(K_0)$, $L:=F^{-1}(L_0)$ and $\psi:=u_k\circ F^{-1}$ and verify that all the conditions (1) to (3) in Lemma~\ref{lem-sb} are satisfied. Thus, it follows from this lemma that there are a point $p$ in $\partial \Omega$, a \nbh\ $V$ of $p$ and a strictly $(q-1)$-\psh\ \fct\ on $V$ such that the set $G:=\{z \in V : \vphi(z)<0\}$ is the desired strictly Levi $(q-1)$-\psc\ set in $V$, whose boundary $\partial G$ shares only a single point with $\partial \Omega$ in $V$.
\end{proof}


\section{On $q$-Pseudoconcave Graphs}\label{chap-main}

\label{sect-qpscve-graphs}

In this section, we will analyze whether submanifolds or graphs of \cont\ functions admit local complex foliations under the condition that their complements are \qpsc. The goal is to generalize the 
classical Hartogs theorem and the results in \cite{Sh93} and \cite{Ch} mentioned in the introduction.

In order to simplify our notations, we introduce a generalized version of concavity based on \qpscy\ in the sense of Rothstein, which we used to call \textit{Hartogs \qpscy} in Chapter \ref{chap-qpsc} (recall Definitions \ref{HaFi} and \ref{defn:relqpsc}).

\bdefn Let $q \in \{1,\ldots,N-1\}$ and let $S$ be a closed subset of an open set $\Omega$ in $\cbb^N$. We say that $S$ is a \textit{(Hartogs) \qpscv{} subset of $\Omega$} if $\Omega':=\Omega\setminus S$ is Hartogs \qpsc{} relatively to $\Omega$. \edefn

As the subsequent statement demonstrates, a surface filled by complex $q$-dimensional analytic sets is strongly related to Hartogs $q$-pseudoconvexity with the same index $q$.

\bprop \label{conv-main} Let $q \in \{1,\ldots,N-1\}$ and let $S$ be a closed subset of an open set $\Omega$ in $\cbb^N$. Assume that the boundary $\partial_\Omega S$ of $S$ in $\Omega$ is locally filled by $q$-dimensional analytic sets, i.e., for every point $p$ in $\partial_\Omega S$ there is a \nbh\ $W$ of $p$ in $\Omega$ such that for each point $w$ in $\partial_\Omega S \cap W$ there exists a $q$-dimensional analytic subset $A_w$ of $W$ with $w \in A_w \subset S$ and $\dim_z A_w \geq q$ for each $z \in A_w$. Then $S$ is a \qpscv\ subset of $\Omega$.
\eprop

\bproof Assume that the statement is false, so the complement of $S$ is not Hartogs \qpsc{} relatively to $\Omega$. Then, according to Theorem \ref{lemNotpsc}, there exist a boundary point $p$ of $S$ in $\Omega$, a \nbh\ $V$ of $p$ and a strictly $(q-1)$-\psc\ set $G$ in $V$ such that $S\cap V$ touches $\partial G$ from the inside of $G$ exactly in $p$. In view of Remark~\ref{rem-def-fct}, we can construct a strictly $(q-1)$-\psh\ \fct\ $\psi$ on some \nbh\ $U$ of $p$ in $V$ which defines $G$ near $p$. By the assumption made on $\partial_\Omega S$, there are a \nbh\ $W$ of $p$ and a $q$-dimensional analytic subset $A$ of $W$ with $p \in A \subset S$. But then $\psi(p)=0$ and $\psi<0$ on $A \cap U$ outside $p$, which contradicts the local maximum principle (see Theorem~\ref{prop-loc-max-analyt}). Therefore, $S$ has to be a $q$-\pscv{} subset of $\Omega$.
\eproof

We present a converse statement on the complex foliation of \qpscv\ CR-submanifolds. For this we need to extend Definition \ref{defn-levi-qpsc} as follows.

\bdefn \label{defn-levi-null} Let $\Gamma = \{\vphi_1=\ldots=\vphi_r=0\}$ be a $\ccal^2$-smooth submanifold in $\cbb^N$ such that $\nabla \vphi_j \neq 0$ on $\Gamma$ for each $j=1,\ldots,r$.
\benum
\item The holomorphic tangent space $H_p \Gamma$ to $\Gamma$ at some point $p$ in $\Gamma$ is given by
\[
H_p \Gamma := \bigcap_{j=1}^r \Big\{ X \in \cbb^N : (\partial \vphi_j(p),X)=\sum_{l=1}^N \frac{\partial\vphi_j}{\partial z_l}(p) X_l =0\Big\}.
\]
\item If the complex dimension of $H_p\Gamma$ has the same value $d$ at each point $p$ in $\Gamma$, then we say that $\Gamma$ is a \textit{CR-submanifold}.

\item The \textit{Levi null space} \index{Levi!null space} of $\Gamma$ at $p$ is the set
\[
\ncal_p:=\bigcap_{j=1}^r\Big\{X \in H_p\Gamma : \mathcal{L}_{\vphi_j}(p)(X,Y)=0 \fe Y \in H_p \Gamma\Big\}.
\]
\eenum
\edefn

\bprop\label{prop-k2} Let $\Gamma = \{\vphi_1=\ldots=\vphi_r=0\}$  be a real $\ccal^2$-smooth CR-submanifold of some open set $\Omega$ in $\cbb^N$ of codimension $r \in \{1,\ldots,2N-1\}$ and fix a number $q \in \{1,\ldots, N-1\}$. Assume further that the complex dimension of the holomorphic tangent space to $\Gamma$ at every point in $\Gamma$ equals $q$ and that $\Gamma$ is a \qpscv{} subset of $\Omega$. Then $\Gamma$ is locally foliated by complex $q$-dimensional submanifolds.
\eprop

\bproof Since the Levi null space lies inside the holomorphic tangent space to~$\Gamma$, it is clear that its complex dimension does not exceed $q$. We claim that the complex dimension of $\ncal_p$ is equal to $q$ for each point $p \in \Gamma$, so that $\ncal_p$ coincides with $H_p \Gamma$.

In order to verify the claim we argue by contradiction and assume that there is a point $p$ in $\Gamma$ such that $\ncal_p$ is a proper subspace of $H_p \Gamma$. This implies that there is an index $j_0$ in $\{1,\ldots,r\}$ and a vector $X_0$ in $H_p \Gamma$ such that $\mathcal{L}_{\vphi_{j_0}}(p,X_0,X_0) \neq 0$. Indeed, a priori, if $\ncal_p \subsetneq H_p\Gamma$, there exist two vectors $X'$ and $Y'$ in $H_p \Gamma$ such that
\[
\mathcal{L}_{\vphi_{j_0}}(p)(X',Y') \neq 0.
\]
If $\mathcal{L}_{\vphi_{j_0}}(p,X',X') \neq 0$ or $\mathcal{L}_{\vphi_{j_0}}(p,Y',Y') \neq 0$, we are done and proceed by picking $X_0=X'$ or, respectively, $X_0=Y'$. Otherwise, if $\mathcal{L}_{\vphi_{j_0}}(p,X',X')$ and $\mathcal{L}_{\vphi_{j_0}}(p,Y',Y')$ both vanish, we can choose an appropriate complex number $\nu$ which satisfies
\[
\mathcal{L}_{\vphi_{j_0}}(p,X'+\nu Y',X'+\nu Y') = 2\repa\big(\nu\mathcal{L}_{\vphi_{j_0}}(p)(X',Y')\big) \neq 0.
\]
Then we continue with $X_0:=X'+\nu Y'$. Now without loss of generality we can assume that $j_0=1$ and $\mathcal{L}_{\vphi_1}(p,X_0,X_0) > 0$. For a positive constant $\mu$ we define another \fct\
\[
\vphi:=\vphi_1 + \mu\sum_{j=1}^r \vphi_j^2.
\]
Since, by the assumptions on $\Gamma$, the gradients $\nabla\vphi_{1},\ldots,\nabla\vphi_{r}$ do not vanish at $p$, there is a \nbh\ $U$ of $p$ such that the set $S:=\{z \in U : \vphi(z)=0\}$ is a real hypersurface containing $\Gamma \cap U$, so that $H_p \Gamma$ becomes a subspace of $H_p S$. Moreover, for $X \in H_p\Gamma$ we can easily compute the Levi form of $\vphi$ at $p$,
\bea
&\mathcal{L}_{\vphi}(p,X,X) = \mathcal{L}_{\vphi_1}(p,X,X) + 2\mu \underbrace{\sum_{j=1}^r |(\partial \vphi_j(p),X)|^2}_{=:R(p,X)}.& \label{eq-R}
\eea
We assert that $H_p S$ contains an $(N-q)$-dimensional subspace $E$ on which $\mathcal{L}_\vphi(p,\cdot)$ is positive definite. To see this, consider the complex normal space $N_p\Gamma$ to $H_p \Gamma$ in $H_p S$,
\[
N_p\Gamma:= \{ Y \in H_pS : \sum_{l=1}^N Y_lX_l=0 \ \fe X \in H_p\Gamma \}.
\]
Observe that $N_p\Gamma$ has dimension $d:=N-q-1$ and choose a basis $Y_1,\ldots,Y_{d}$ of $N_p\Gamma$. Let $E$ be the complex span of the vectors $X_0$ from above and $Y_1,\ldots,Y_{d}$. Since $X_0$ belongs to $H_p \Gamma$, but $Y_1,\ldots,Y_{d}$ do not, the dimension of $E$ equals $N-q$.

We set $E_0:=\{ Z \in E : |Z|=1\}$ and $M:=\{Z \in E_0 : \mathcal{L}_{\vphi_1}(p,Z,Z) \leq 0\}$.

If $M$ is empty, then $\mathcal{L}_{\vphi_1}(p,\cdot,\cdot)$ is positive on $E$ and we can put $\mu=0$.

If $M$ is not empty, notice first that, if $Z$ lies in $M$, then $R(p,Z)>0$ (recall the equation \eqref{eq-R} for the definition of $R(p,Z)$). Otherwise $Z$ belongs to $H_p \Gamma$ and, therefore, it is a multiple of $X_0$, i.e., $Z=\lambda X_0$ for some complex number~$\lambda$. But then $\mathcal{L}_{\vphi_1}(p,Z,Z)=|\lambda|^2\mathcal{L}_{\vphi_1}(p,X_0,X_0)>0$ and $Z$ lies in $M$ at the same time, which is absurd. Hence, $R(p,Z)>0$ for every vector $Z$ in $M$.  Since $E_0$ is compact and $\Gamma$ is $\ccal^2$-smooth, we can find constants $c_0>0$ and $c_1>0$ such that $R(p,Z)\geq c_0$ for every $Z$ in $M$ and $\mathcal{L}_{\vphi_1}(p,Z,Z) \geq -c_1$  for every $Z$ in $E_0$. Now we can choose $\mu$ so large that $-c_1 + \mu c_0 >0$ in order to obtain that $\mathcal{L}_{\vphi}(p,Z,Z)>0$ for each $Z$ in $E_0$. Since $\mathcal{L}_{\vphi}(p,\lambda X,\lambda X)=|\lambda|^2 \mathcal{L}_{\vphi}(p, X,X)$ for every $X$ in $E$ and $\lambda$ in $\cbb$, we have that $\mathcal{L}_{\vphi}(p, \cdot,\cdot )$ is positive on $E\setminus\{0\}$.

Therefore, in both cases, the Levi form $\mathcal{L}_{\vphi}$ at $p$ is positive definite on the $(N-q)$-dimensional space $E$ in $H_pS$. Hence, $\{\vphi<0\}$ is strictly $(q-1)$-\psc\ at $p$. Since $\Gamma \cap B \subset \partial\{\vphi<0\}$ for some open ball $B$ centered at $p$, it follows from the continuity princple in Theorem~\ref{lemCofspsc} and Theorem~\ref{equivqpsc} that the complement of $\Gamma$ cannot be Hartogs \qpsc\ relatively to $B$. Therefore, it cannot be a $q$-\pscv\ subset of $\Omega$, which is a contradiction to the assumption made on $\Gamma$. Finally, we can conclude that $\ncal_p=H_p \Gamma$. By assumption, these two spaces have constant dimension $q$ on $\Gamma$, so Theorem~1.1 in \cite{Free} implies that $\Gamma$ admits a local foliation by complex $q$-dimensional submanifolds.
\eproof

Locally, a generic smooth submanifold of codimension $k$ in $\cbb^{N}$ can be represented by the graph of a smooth mapping $f:U \to \rbb^k$, where $U$ is an open set in $\cbb^n\times\rbb^k$ with $n=N-k$ (see e.g.~\cite[Chapter §1.3]{BER}). This representation is not unique in general, but it motivates to investigate conditions under which the graph of a continuous mapping $f:U \to \rbb^k\times \cbb^p$ carries a foliation by complex submanifolds.

\begin{setting} \label{set-qpscv} Fix integers $n\geq 1$ and $k,p \geq 0$ such that $N=n+k+p \geq 2$. Then $\cbb^N$ splits into the product
\[
\cbb^N=\cbb^N_{z,w,\zeta}=\cbb^n_z\times \cbb^k_w \times \cbb^p_\zeta=\cbb^n_z\times (\rbb_u^k+i\rbb_v^k) \times \cbb^p_\zeta,
\]
where $w=u+iv$. Let $G$ be an open set in $\cbb^n_z\times\rbb^k_u$ and let $f=(f_v,f_\zeta)$ be \cont\ on $G$ with image in $\rbb_v^k\times\cbb^p_\zeta$. Then the graph of $f$ is given by
\[
\Gamma(f)=\{(z,w,\zeta) \in \cbb^n_z\times \cbb^k_w \times \cbb^p_\zeta: \ (z,u) \in G, \ (v,\zeta)=f(z,u)\}.
\]
Moreover, we denote by $\pi_{z,u}$ the natural projection
\[
\pi_{z,u} : \cbb_{z}^n\times \cbb_{w}^{k} \to \cbb_{z}^n\times\rbb_{u}^k, \quad \pi_{z,u}(z,w) \mapsto (z,u).
\]
\end{setting}

We are interested in the question whether $\Gamma(f)$ admits a local foliation by complex submanifolds. In this context, we first have to study the \qpscvy\ of the graph of $f$.

\begin{lem} \label{lemGraphs} Pick another integers $m \in\{ 1,\ldots,n\}$ and $r \in \{ 0,\ldots,p\}$ with $k+ r \geq 1$. For $\mu_1,\ldots,\mu_r \in \{ 1,\ldots,p\}$ with $\mu_1<\ldots<\mu_r$, we divide the coordinates of $\zeta$ into $\zeta'=(\zeta_{\mu_1},\ldots,\zeta_{\mu_r})$ and the remaining coordinates $\zeta''=(\zeta_{\mu_j}:j \in \{1,\ldots p\}\setminus\{\mu_1,\ldots,\mu_r\})$ which we assume to be ordered by their index $\mu_j$, as well. Finally, let $\Pi$ be a complex $m$-dimensional affine subspace in $\cbb^n_{z}$. We set $M=m+k+r$ and $\cbb^M_\bullet:=\Pi\times\cbb_w^k\times\cbb^r_{\zeta'}$, $G_\bullet:=G \cap (\Pi\times\rbb^k_u)$ and $f_\bullet:=(f_v,f_{\zeta'})|_{G_\bullet}$.

If the graph $\Gamma(f)$ is an $n$-\pscv{} subset of $G\times\rbb^k_v\times\cbb^p_\zeta$, then the graph $\Gamma(f_\bullet)$ is an $m$-\pscv{} subset of $G_\bullet\times\rbb^k_v\times\cbb^r_{\zeta'}$.
\end{lem}

\begin{proof} Since the Hartogs $n$-\pscy{} is a local property according to Theorem~\ref{equivqpsc}, after shrinking $G$ if necessary and after a biholomorphic change of coordinates we can assume without loss of generality that $\Pi=\{0\}^{n-m}_{z'}\times\cbb^m_{z''} \subset \cbb_z^n$, where $z'=(z_1,\ldots,z_{n-m})$ and $z''=(z_{n-m+1},\ldots,z_n)$, and that the $\zeta$-coordinates are ordered in such a way that $\zeta'=(\zeta_1,\ldots,\zeta_r)$ and $\zeta''=(\zeta_{r+1},\ldots,\zeta_p)$.

Assume that the complement of $\Gamma(f_\bullet)$ is not Hartogs $m$-\psc{} relatively to $G_\bullet\times\rbb^k_v\times\cbb^r_{\zeta'}$ and recall that $M=m+k+r$. Then, by Definition~\ref{defn:relqpsc}, there are a point $p$ in $\Gamma(f_\bullet)$ and a ball $B$ centered at $p$ in $\cbb^M$ such that the set $(\cbb^M\setminus \Gamma(f_\bullet))\cap B$ is not Hartogs $m$-\psc. Since $B$ is strictly \psc, according to Theorem~\ref{eqivqpsc}~\ref{eqivqpsc7} there is a family $\{A_t\}_{t\in[0,1]}$ of analytic sets $A_t$ in $\cbb^M$ which depends continuously on $t$ and which fulfills the following properties:
\begin{itemize}

\item $\dim_z A_t \geq M-m = k+r$ for every $z \in A_t$ and $t \in [0,1]$, 

\item the closure of the union $\D \bigcup_{t \in [0,1]} A_t$ is compact,

\item for every $t \in [0,1)$ the intersection $\cl{A_t} \cap \Gamma(f_\bullet)$ is empty,

\item $\partial A_1 \cap \Gamma(f_\bullet)$ is empty, as well,

\item and the set $A_1$ touches $\Gamma(f_\bullet)$ at a point $p_0=(z_0,w_0,\zeta'_0)$, where $z_0=(z'_0,z_0'')=(0,z_0'')$ and $w_0=u_0+iv_0$.
\end{itemize}
Given some positive number $\rho$, consider the $(k+p)$-dimensional analytic sets,
\[
S_t:=\{0\}^{n-m}\times A^t \times \Delta_\rho^{p-r}(f_{\zeta''}(z_0,u_0)).
\]
It is easy to verify that the family $\{S_t\}_{t \in [0,1]}$ of $(k+p)$-dimensional analytic sets violates the continuity principle of Theorem~\ref{eqivqpsc} applyied to the complement of $\Gamma(f)$ in $\CC^N=\CC^{n+k+p}$. Hence, $\Gamma(f)$ cannot be Hartogs $n$-\pscv{} relatively to $G\times\RR^k_v\times\CC^p_\zeta$, which is a contradiction to the assumption made on $\Gamma(f)$. Finally, this means that $\Gamma(f_\bullet)$ has to be an $m$-\pscv{} subset of $G_\bullet\times\rbb^k_v\times\cbb^r_{\zeta'}$.
\end{proof}

We are now able to prove the main theorem.

\begin{thm}\label{thm-pscv-fol} Let $n,k,p$ be integers with $n \geq 1$, $p \geq 0$ and $k \in \{0,1\}$ such that $N=n+k+p \geq 2$. Let $G$ be a domain in $\cbb^n_z\times\rbb^k_u$ and let $f: G \nach \rbb^k_v \times \cbb^p_\zeta$ be a \cont \ \fct\ such that $\Gamma(f)$ is an $n$-\pscv\ subset of $G\times\RR^k_v\times\CC^p_\zeta$. Then $\Gamma(f)$ is locally the disjoint union of $n$-dimensional complex submanifolds.
\end{thm}

\begin{proof} The statement is of local nature, so we can assume that $G$ is an open ball in $\cbb_z^n\times\rbb^k_u$ and that $\Gamma(f)$ is bounded. Recall also that Hartogs $n$-\pscy\ in $\CC^{n+1}$ is equivalent to the classical \pscy. 

We separate the problem into the subsequent cases.

\paragraph{\textbf{Case} $\mathbf{n \geq 1}$, $\mathbf{k=0}$, $\mathbf{p=1}$.} This is the classical Hartogs' theorem (see Theorem~\ref{thm-hartogs-pscv} in the Introduction).

\paragraph{\textbf{Case} $\mathbf{n \geq 1}$, $\mathbf{k=0}$, $\mathbf{p\geq 1}$.} For each $j \in \{1,\ldots,p\}$ the graph $\Gamma(f_{\zeta_j})$ is an $n$-\pscv\ subset of $G\times\CC_{\zeta_j}$ in $\CC_{z,\zeta_j}^{n+1}$ due to Lemma~\ref{lemGraphs}, so its complement is pseudoconvex in the classical sense. By Hartogs' theorem, the functions $f_j$ and, therefore, the mapping $f=(f_1,\ldots,f_p)$ are holomorphic which means that $\Gamma(f)$ is an $n$-dimensional complex manifold.

\paragraph{\textbf{Case} $\mathbf{n = 1}$, $\mathbf{k=1}$, $\mathbf{p=0}$.} This was proved by the second author in \cite{Sh93}.

\paragraph{\textbf{Case} $\mathbf{n \geq 1}$, $\mathbf{k=1}$, $\mathbf{p=0}$.} This case has been treated by Chirka in \cite{Ch}.

\paragraph{\textbf{Case} $\mathbf{n = 1}$, $\mathbf{k=1}$, $\mathbf{p = 1}$.} By Lemma~\ref{lemGraphs} the graph $\Gamma(f_v)$ is a $1$-\pscv{} subset of $G\times\RR_v$ in $\cbb^2_{z,w}$, i.e.,~its complement is pseudoconvex in the classical sense. According to \cite{Sh93}, it is locally foliated by a family of holomorphic curves $\{\gamma_\alpha\}_{\alpha \in I}$ represented as graphs of holomorphic functions $g_\alpha$ which are all defined on a disc $D$ in $\cbb_z$ that does not depend on the indexes $\alpha \in I$. Denote by $\pi_z$ the standard projection of points in $\cbb_{z,w}^2$ into $\cbb_z$. We define mappings $f_\zeta^\alpha$ by the assignment
\[
\gamma_\alpha \ni t \ \mapsto \ f_\zeta^\alpha(t):=f_\zeta\big(\pi_z(t),\mathrm{Re}(g_\alpha)(\pi_z(t))\big).
\]
Since for each $\alpha \in I$ the curve $\gamma_\alpha$ is represented by the graph $\Gamma(g_\alpha)$, the \fct\ $f^\alpha_\zeta:\gamma_\alpha \to \cbb_\zeta$ is well-defined, and its graph is given by $\Gamma(f^\alpha_\zeta)=\Gamma(f|_{\pi_{z,u}(\gamma_\alpha)})$. Here, $\pi_{z,u}$ means the standard projection of $\cbb^2_{z,w}$ to $\cbb_z\times\rbb_u$. We claim that $\Gamma(f^\alpha_\zeta)$ is indeed a holomorphic curve.

In order to verify the claim we argue by contradiction and assume that there is an $\alpha_0$ such that the mapping $f_\zeta^{\alpha_0}$ is not holomorphic in a \nbh \ of a point $(z_0,w_0) \in \gamma_{\alpha_0}$. After a local holomorphic change of coordinates, we can assume that $\gamma_{\alpha_0}=\Delta_r(z_0) \times \{ w=0\}$, where $\Delta_r(z_0)\relc D$ is a disc in $\cbb_z$ centered in $z_0$. After a reparametrization we can arrange that $\alpha_0=0$ and $(-1,1)\subset I$. Since the curve $\gamma_{0}$ is of the form $\Delta_r(z_0) \times \{ w=0\}$ near $z_0$, we can treat $f_\zeta^{0}$ as a function $f^0_\zeta:\Delta_r(z_0) \nach \cbb_\zeta$. By our assumption that $f^0_\zeta$ is not holomorphic, in view of Hartogs' Theorem \ref{thm-hartogs-pscv}, it follows that the set $\Gamma(f^0_\zeta)$ is not a $1$-\pscv\ subset of $\Delta_r(z_0)\times\CC_\zeta$ in $\cbb^2_{z,\zeta}$. Then, by Theorem~\ref{lemNotpsc}, there exist a point $p_1=(z_1,\zeta_1) \in \Gamma(f^0_\zeta)$, a small enough open \nbh\ $V$ of $p_1$ in $\cbb^2_{z,\zeta}$, a $\ccal^2$-smooth \spsh\ \fct\ $\vrho_1=\vrho_1(z,\zeta)$ on $V$ with $\nabla\vrho_1\neq 0$, and radii $\sigma,r',r''>0$ with $r''<r'<r$ such that
\bea
& \Delta_{r'}\times\Delta_\sigma \relc V, \qquad \Gamma(f^0_\zeta|_{\cl{\Delta_{r'}}}) \subset \cl{\Delta_{r'}\times\Delta_\sigma},& \nonumber\\
\label{eq-pscv-gr-000} \\
& \Gamma(f^0_\zeta|_{\cl{\Delta_{r'}}}) \subset \{\vrho_1 \leq 0\}, \qquad \Gamma(f^0_\zeta|_{\cl{\Delta_{r'}}}) \cap \{\vrho_1= 0\} = \{(z_1,\zeta_1)\},& \nonumber
\eea
\bea
& \and \ \Gamma(f^0_\zeta|_{\cl{A_{r',r''}}}) \cap \{\vrho_1= 0\} = \emptyset,& \label{eq-pscv-gr-001}
\eea
where each disc $\Delta_s$ in $\cbb_z$ mentioned above is assumed to be centered at $z_1$, $A_{r',r''}:=\Delta_{r''}\setminus\cl{\Delta_{r'}}$ and the disc $\Delta_\sigma \subset \cbb_{\zeta}$ is assumed to be centered at $\zeta_1$. For $\alpha \in (-1,1)$ we set $\gamma'_\alpha:=\Gamma(g_\alpha|_{\overline{\Delta_{r'}}})$ and $\Gamma'_\alpha:=\Gamma(f^\alpha_\zeta|_{\gamma'_\alpha})$. Since $f$ is \cont, and since the family $\{\gamma_\alpha\}_{\alpha \in I}$ depends continuously on $\alpha$, it follows from~\eqref{eq-pscv-gr-001} that there is a number $\tau \in (0,1)$ such that
\begin{eqnarray}
& K:=\bigcup_{\alpha \in [-\tau,\tau]} \Gamma_\alpha' \ \subset \ \Delta_{r'}\times\cbb_w\times\Delta_{\sigma}& \nonumber\\
\label{eqrho2}\\
& \and \ \varrho_1 < 0 \ \mathrm{on} \  \Gamma'_\alpha\cap \left(\cl{A_{r',r''}} \times \cbb^2_{w,\zeta} \right) \fe \alpha \in [-\tau,\tau],& \nonumber
\end{eqnarray}
where $\vrho_1$ is now considered as a \fct\ defined on $\{(z,w,\zeta) \in \cbb^3 : (z,\zeta) \in V\}$.

Since the curves in the familiy $\{\gamma_\alpha\}_{\alpha \in I}$ are holomorphic, the set
\[
A:=\big(\gamma_{-\tau}\cup\gamma_{\alpha_0}\cup\gamma_{\tau}\big) \cap (\Delta_{r}\times\cbb_w)
\]
is a closed analytic subset of the \psc\ domain $\Delta_{r}\times\cbb_w$. Let $h$ be a holomorphic \fct\ on $A$ defined by $h\equiv 0$ on $\gamma_{\pm\tau}$ and $h \equiv 1$ on $\gamma_{\alpha_0}$. Then there exists a holomorphic extension $\hat{h}$ of $h$ into the whole of $\Delta_{r}\times\cbb_w$ (see Theorem 4 in Paragraph 4.2 of Chapter V in \cite{GrRe}). Hence, the \fct\ $\vrho_2(z,w):=\log|\hat{h}(z,w)|$ is \psh\ on $\Delta_{r}\times\cbb_w$ and satisfies $\vrho_2\equiv -\infty$ on $\gamma_{\pm \tau}$. Now for $\eps>0$ we define
\[
\psi_0(z,w,\zeta):=\varrho_1(z,\zeta) + \eps \varrho_2(z,w),
\]
where $\vrho_1$ is the defining \fct\ from above. By the inequality (\ref{eqrho2}) and the properties of $\vrho_2$, for a sufficiently small $\eps>0$ we obtain that
\bea
\psi_0 <0 \ \mathrm{on} \ L:= \bigcup_{\alpha \in [-\tau,\tau]}\big(\Gamma'_\alpha\cap \left(\cl{A_{r',r''}} \times \cbb^2_{w,\zeta} \right)\big) \cup \Gamma'_{\tau} \cup \Gamma'_{-\tau}.\label{eqrho3}
\eea
By the choice of the point $(z_1,\zeta_1)$ above and by the inclusion $(z_1,0) \in \gamma_{\alpha_0}$, we have that $\vrho_1(z_1,\zeta_1)=0$, $\vrho_2(z_1,0)=0$ and, therefore, $\psi_0(z_1,0,\zeta_1)=0$. Since $(z_1,0,\zeta_1)$ belongs to $K$, it follows from the inequality~\eqref{eqrho3} that $\psi_0$ attains a non-negative maximal value on $K$ outside $L$. Since $\psi_0$ is \psh\ on a \nbh\ of $K$, by using standard methods on the approximation of \psh\ functions we can assume without loss of generality that $\psi_0$ is $\ccal^\infty$-smooth and strictly \psh\ on a \nbh\ of $K$, satisfies the property \eqref{eqrho3} and still attains its maximum on $K$ outside $L$.

Now it is easy to verify that $K, L$ and $\psi:=\psi_0-\max_K \psi_0$ fulfill all the conditions (1) to (3) of Lemma~\ref{lem-sb}. Thus, there exist a point $p_2$ in $K\setminus L$, a \nbh\ $U$ of $p_2$ containing $K$ and $L$ and a $\ccal^2$-smooth \spsh\ \fct\ $\vphi$ on $U$ so that $U':=\{(z,w,\zeta) \in U : \vphi(z,w,\zeta)<0\}$ is strictly \psc\  relatively to $U$, $\vphi<0$ on $L$, $\vphi \leq 0$ on $K$ and $\vphi(z,w,\zeta)$ vanishes on $K$ if and only if $(z,w,\zeta)=p_2$. Since $U'$ is \spsc\ at $p_2$, we derive from Theorem~\ref{lemCofspsc} and the continuity principle in Theorem~\ref{equivqpsc} that the graph $\Gamma(f)$ cannot be a $1$-\pscv{} subset of $G\times\RR_v\times\CC_\zeta$ in $\cbb^3_{z,w,\zeta}$, which is a contradiction to the assumption made on $\Gamma(f)$. As a conclusion, the curves in $\{\Gamma(f^\alpha_\zeta)\}_{\alpha \in I}$ have to be holomorphic. This leads to the desired local complex foliation of $\Gamma(f)$.

\paragraph{\textbf{Case} $\mathbf{n \geq 1 }$, $\mathbf{k=1}$, $\mathbf{p=1}$.} According to Lemma~\ref{lemGraphs} with $m=n$ and $r=0$, the graph $\Gamma(f_v)$ is an $n$-\pscv{} subset of $G\times\rbb_v$ in $\CC^{n+1}_{z,u+iv}$, where $G$ is a ball in $\cbb^n_z\times\rbb_u$. Hence, by Chirka's result (see the case $n \geq 1$, $k=1$, $p=0$), the graph $\Gamma(f_v)$ is foliated by a family $\{A_\alpha\}_{\alpha \in I}$ of holomorphic hypersurfaces $A_\alpha$. For $\alpha \in I$ define the \fct
\[
f^\alpha_\zeta : A_\alpha \nach \cbb_\zeta \quad \mathrm{by} \quad f^\alpha_\zeta = f_\zeta|_{\pi_{z,u}(A_\alpha)}
\]
and identify $\Gamma(f^\alpha_\zeta)$ with $\Gamma(f|_{\pi_{z,u}(A_\alpha))}$. Suppose that some \fct\ $f^{\alpha_0}_\zeta$ is not holomorphic. Then, by Hartogs' theorem of separate holomorphicity, there is a complex one-dimensional curve $\sigma_{\alpha_0}$ in $A_{\alpha_0}$ on which $f^{\alpha_0}_\zeta$ is not holomorphic near a point $p_0 \in \sigma_{\alpha_0}$. After a change of coordinates we can assume that $p_0=0$, $f^{\alpha_0}_\zeta(0)=0$ and $\sigma_{\alpha_0}=\Delta\times\{ z_2=\ldots=z_n=w=0\}$ in a \nbh \ of 0, where $\Delta$ is the unit disc in $\cbb_{z_1}$. We set $\lbb:=\cbb_{z_1}\times\{0\}^{n-1}$. By Lemma~\ref{lemGraphs} the graph $\Gamma(f_\bullet)$ of $f_\bullet:=f|_{G_\bullet}$ is a 1-\pscv{} subset of $G_\bullet\times\RR_v\times\CC_\zeta$ in $\cbb^3_{z_1,w,\zeta}$, where $G_\bullet:= (G\cap(\lbb\times\rbb_u))$. Thus, in view of the considered above case $n=k=p=1$, the graph $\Gamma(f_\bullet)$ is foliated by complex curves of the form
\[
(f_\bullet)^\beta_\zeta:\gamma_\beta \nach \cbb_\zeta \quad \mathrm{with} \quad (f_\bullet)_\zeta^\beta=(f_\bullet)_\zeta|_{\pi_{z,u}(\gamma_\beta)},
\]
where $\{\gamma_\beta\}_{\beta \in I}$ is a family of holomorphic curves of a foliation of $\pi_{z_1,w}(\Gamma(f_\bullet))$. From the uniqueness of the foliation on $\pi_{z_1,w}(\Gamma(f_\bullet))$ we deduce that $\pi_{z_1,w}(\sigma_{\alpha_0})$ coincides (at least locally) with a curve $\gamma_{\beta_0}$ containing 0. Hence, in some \nbh\ of 0 we have that $\gamma_{\beta_0}=\Delta\times\{ 0\}$ and therefore
\begin{eqnarray*}
f^{\alpha_0}_\zeta |_{\sigma_0} & = & f_\zeta|_{\Delta\times\{ z_2=\ldots=z_n=0\}\times\{u=0 \}}\\
& = & (f_\bullet)_\zeta|_{\Delta\times\{u=0\}} = (f_\bullet)_\zeta|_{\pi_{z,u}(\gamma_{\beta_0})} = (f_\bullet)_\zeta^{\beta_0}.
\end{eqnarray*}
This means that $f^{\alpha_0}_\zeta$ has to be holomorphic on a \nbh \ of 0 in $\sigma_{\alpha_0}$, which is a contradiction to the choice of $f_\zeta^{\alpha_0}$ and $\sigma_{\alpha_0}$. Hence, $\{\Gamma(f^\alpha_\zeta)\}_{\alpha \in I}$ is the desired foliation of~$\Gamma(f)$.

\paragraph{\textbf{Case} $\mathbf{n = 1}$, $\mathbf{k=1}$, $\mathbf{p \geq 1}$.} We derive from Lemma~\ref{lemGraphs} with $m=1$ and $r=0$ that the graph $\Gamma(f_v)$ is a 1-\pscv{} subset of $G\times\RR_v$ in $\cbb^2_{z,w}$. It follows then from the theorem of the second author (see case $n=k=1$, $p=0$ above) that the graph $\Gamma(f_v)$ is foliated by the family $\{\gamma_\alpha\}_{\alpha \in I}$ of holomorphic curves~$\gamma_\alpha$. Define similarly to the previous cases for $\alpha \in I$ the mapping
\bea \label{case-last}
f_\zeta^\alpha=(f^\alpha_{\zeta_1},\ldots,f^\alpha_{\zeta_p}) : \gamma_\alpha \nach \cbb_{\zeta}^p \quad \mathrm{by} \quad f^\alpha_\zeta := f_\zeta|_{\pi_{z,u}(\gamma_\alpha)}.
\eea
Since for each $j \in \{1,\ldots,p\}$ the graph $\Gamma(f_v,f_{\zeta_j})$ is a $1$-\pscv{} subset of $G\times\RR_v\times\CC_{\zeta_j}$ in $\cbb^3_{z,w,\zeta_j}$ due to Lemma~\ref{lemGraphs} with $m=1$ and $r=1$, it follows by the same arguments as in the case $n=k=p=1$ that the component $f^\alpha_{\zeta_j}:\gamma_\alpha \nach \cbb_{\zeta_j}$ is holomorphic. Hence, the curve $f^{\alpha}_\zeta$ is holomorphic, as well, so that $\Gamma(f)$ is foliated by the family $\{\Gamma(f_\zeta^\alpha)\}_{\alpha\in I}$ of holomorphic curves.

\paragraph{\textbf{Case} $\mathbf{n \geq 1}$, $\mathbf{k=1}$, $\mathbf{p \geq 1}$.} The proof is nearly the same as in the previous case $n=k=1$, $p \geq 1$. We only need to replace the curves $\{\gamma_\alpha\}_{\alpha \in I}$ in \eqref{case-last} by complex hypersurfaces $\{H_\alpha\}_{\alpha \in I}$ obtained from Chirka's result (case $n \geq 1$, $k=1$, $p=0$) and to apply the case $n\geq 1$, $k=1$, $p=1$ to each $j=1,\ldots,p$ in order to show that $f^\alpha_{\zeta_j}:H_\alpha \to \cbb_{\zeta_j}$ is holomorphic on $H_\alpha$. Then $\{\Gamma(f^\alpha_\zeta)\}_{\alpha \in I}$ is a complex foliation of~$\Gamma(f)$.

The proof of our main theorem is finally complete.
\end{proof}

So far, we do not have techniques to treat the case $n=1$, $k=2$ and $p \geq0$. The next example from \cite{JKS} shows that it is not always possible to foliate a 1-\pscv\ real 4-dimensional submanifold in $\cbb^3$ by complex submanifolds, but it is still possible to do this by analytic subsets.\footnote{Thanks to Prof. Kang-Tae Kim for attracting our attention to this example.}

\bex For a fixed integer $k \geq 2$ consider the \fct\
\[
f(z_1,z_2):=\left\{ \begin{array}{ll} \ol{z}_1 {z_2}^{2+k}/\ol{z}_2, & z_2 \neq 0 \\ 0, & z_2=0 \end{array}\right. .
\]
It is $\ccal^k$-smooth on $\cbb^2$ and holomorphic on complex lines passing through the origin, since $f(\lambda v)=\lambda^{2+k}f(v)$ for every $\lambda \in \cbb^*:=\cbb{\setminus}\{0\}$ and each vector $v \in \cbb$. Therefore, in view of \cite{Ba}, the \fct\ $f$ is 1-\hol\ on $\cbb^2$, so $\psi(z_1,z_2,w):=-\log|f(z_1,z_2)-w|$ is 1-\psh\ outside $\{ f=w\}$ by \cite{HM}. Due to Theorem \ref{equivqpsc}, this means that the graph $\Gamma(f)$ of $f$ is a 1-\pscv\ real 4-dimensional submanifold of $\cbb^3$ which does not admit a regular foliation near the origin, but admits a singular one which is given by the family of holomorphic curves $\{\Gamma(f|_{\cbb^*v)}\}_{v \in \cbb^2}$. Of course, the problem arises because the complex Jacobian of $f$ has non-constant rank near the origin.
\eex

\bibliographystyle{alpha}

\end{document}